\documentclass[12pt,reqno]{amsart}

\usepackage{amsmath,amssymb,amsthm,amscd,amsfonts,amsbsy,bm}
\usepackage[english]{babel}
\usepackage[mathscr]{eucal}

\usepackage{graphicx}

\usepackage{xcolor}

\numberwithin{equation}{section}

\renewcommand{\a}{\alpha}
\renewcommand{\b}{\beta}

\newcommand{\e}{\epsilon}

\renewcommand{\k}{\kappa}

\renewcommand{\l}{\lambda}

\newcommand{\n}{\nu}

\newcommand{\F}{\Phi}

\renewcommand{\o}{\omega}
\renewcommand{\O}{\Omega}

\newcommand{\C}{{\mathbb C}}
\newcommand{\R}{{\mathbb R}}

\newcommand{\tb}{{\mathbf t}}

\newcommand{\Fb}{{\mathbf F}}
\newcommand{\Gb}{{\mathbf G}}

\newcommand{\Pb}{{\mathbf P}}

\newcommand{\Tb}{{\mathbf T}}

\newcommand{\cF}{\mathfrak c}

\newcommand{\DF}{\mathfrak D}

\newcommand{\FF}{\mathfrak F}

\newcommand{\kF}{\mathfrak k}
\newcommand{\KF}{\mathfrak K}

\newcommand{\pF}{\mathfrak p}
\newcommand{\PF}{\mathfrak P}

\newcommand{\Ac}{{\mathcal A}}
\newcommand{\Bc}{{\mathcal B}}

\newcommand{\Hc}{{\mathcal H}}

\newcommand{\Kc}{{\mathcal K}}

\newcommand{\vol}{{\rm vol}\,}

\DeclareMathOperator{\re}{{\rm Re}\,}

\newcommand{\dbar}{\overline{\partial}}

\newtheorem{theorem}{Theorem}[section]

\newtheorem{corollary}[theorem]{Corollary}

\theoremstyle{definition}
\newtheorem{definition}[theorem]{Definition}

\theoremstyle{remark}
\newtheorem{remark}[theorem]{Remark}

\begin{document}

\title[Finite Rank Theorem]{The finite rank theorem for Toeplitz operators on the Fock space}
\author[Borichev]{Alexander Borichev}
\address{LATP, Aix-Marseille Universit\'e, Marseille, France}
\email{borichev@cmi.univ-mrs.fr}
\author[Rozenblum]{Grigori Rozenblum}
\address{Department of Mathematics, Chalmers  University of Technology,   Gothenburg, Sweden}
\email{grigori@chalmers.se}

\begin{abstract} We consider Toeplitz operators on the Fock space, under rather general conditions imposed on the symbols. It is proved that if the operator has finite rank  then the operator and the symbol should be zero.  The method of proof is different from the ones used previously for finite rank theorems, and it enables us to get rid of the compact support condition and even allow a certain growth of the symbol.
\end{abstract}
\maketitle
\section{introduction}\label{intro}

The paper is devoted to the study of the finite rank problem for  Toeplitz operators on the Fock (Bargmann--Segal) spaces. Such operators were  introduced by F.~A.~Berezin in \cite{Ber}, in the framework  of his general quantization program, and  were being extensively studied from different points of view further on, see, especially, \cite{Cob2} and the books \cite{Zhu} and \cite{Vasil}. These operators are often called Berezin--Toeplitz ones, and they form a special class of Toeplitz operators in Bergman type spaces.

Properties of Toeplitz operators  on Bergman type spaces attract currently considerable interest, due to an expanding range of applications in Analysis and Mathematical Physics. One of the questions that has been  under discussion recently is the one on finite rank operators.

Generally speaking, let $\Bc\subset L^2(\O)=L^2(\O,d\mu)$, with some measure $d\mu$, be a Bergman type space consisting of solutions of some elliptic equation or system in a domain $\O$ in a real or complex Euclidean space. Denote by $\Pb:L^2(\O)\to \Bc$ the orthogonal projection onto $\Bc$. For a function $F$, the Toeplitz operator $\Tb=\Tb(F)$ is the operator in $\Bc$ acting as $\Bc\ni u\mapsto \Pb F u\in \Bc$. Here, $F$ is called the \emph{symbol} of the Toeplitz operator. This definition is unambiguous for the case of a bounded function $F$. However, the formula defining the action of the operator can be assigned an exact meaning also for certain unbounded functions $F$, for measures and even for some distributions. A detailed description of such Toeplitz operators  can be found in \cite{AlexRoz}, \cite{RozToepl}, \cite{Rao}; we give more explanations below. For  $\Bc$, the most common examples are the space of analytical functions in a bounded domain in $\C^d$ (say, disk, ball, polydisk) -- the classical Bergman spaces, as well as  the space of the entire functions in $\C^d$ square integrable with the Gaussian weight -- the Bargmann--Fock space,  similar spaces of harmonic functions etc.

The finite rank problem consists in the following. Suppose that for some symbol $F$, the operator $\Tb(F)$ has finite rank. What can be said about $F$ in this case?

The starting point here was the paper \cite{Lue1} by D.~Luecking, where this problem was considered for  Toeplitz operators in the analytical Bergman space. It was stated there, see \cite[Lemma 2, p.365]{Lue1}, that for $F$ being a bounded function with compact support, only the trivial case may happen, i.e. the finite rank property implies that $F=0$. Unfortunately, the proof in \cite{Lue1} was not correct. Later, in \cite[Remark, p.515]{Cob} it was stated, without proof, that a Toeplitz operator of rank \emph{one} must be zero. According to the author of \cite{Cob} (private communication), the proof (which, eventually, became a part of the folklore) had been found by him together with C.~Berger, using an idea by R.~Rochberg. One can find an exposition of this short proof, e.g., in \cite{BauLe}.

Only in 2008, a correct proof of the finite rank theorem was found, again by Luecking \cite{Lue2}, using the ideas quite different from the ones of the initial approach as well as of the rank one proof  by Berger--Coburn. Moreover, a more general case has been considered. Under the assumption that $F$ is a \emph{complex finite regular measure} with compact support in $\C^1$, it was proved that the finite rank property for $\Tb(F)$ implies that $F$ is a finite collection of point masses.

Quite soon, a series of papers with generalizations and applications of the finite rank theorem appeared, see \cite{AlexRoz}, \cite{BauLe}, \cite{BRChoe1}, \cite{Le}, \cite{Le1}, \cite{Rao}, \cite{RShir}; an extensive presentation of the majority of the corresponding results can be found in \cite{RozToepl}. Luecking's theorem has been extended to the case of many variables, to Bergman spaces of harmonic functions or solutions of the Helmholtz equation, to symbols being distributions etc. Characteristical for all these papers was that  all of them  were based upon the approach used in \cite{Lue2}. This approach reduces the finite rank problem to a polynomial approximation one, with ultimate use of the Stone-Weierstrass theorem for a certain  algebra of symmetric polynomials. This final circumstance, when Fock--type spaces were being considered, led to the requirement for the symbol to have compact support, as the Stone--Weierstrass theorem needs. It was shown in \cite{BauLe}, \cite{GrVas} that if  the compact support condition is dropped, then the finite rank theorem may become wrong even for rank zero, unless some restrictions on the behavior of the symbol at infinity are imposed. The compact support conditions were relaxed
 in \cite{RozSW}, where for symbols-measures, very fast decaying at infinity, the finite rank theorem has been proved. Again, the approach by Luecking was used, and the result was established by introducing a new version of the Stone--Weierstrass theorem (proved in the same paper) where the uniform convergence of functions on a compact space is replaced by the weighted convergence on a locally compact space.
We mention also that W.~Bauer and T.~Le \cite{BauLe} extended the rank one theorem to symbols-functions satisfying some growth restrictions.

The decay conditions imposed on the symbols in \cite{RozSW}, look obviously excessive; on the other hand, the result in \cite{RozSW} does not apply to symbols-distributions, more singular than measures, which is an essential drawback for applications.
Therefore the attack on the finite rank problem was continued.

In the present paper we  return to the approach  used in \cite{Lue1} and in the proof of the rank one case, implicit in \cite{Cob}. We consider  a rather wide class of symbols-functions with  some mild restrictions on the behavior at infinity, and for such symbols we establish the finite rank theorem for Toeplitz operators in the Fock space.

We start in Section 2 with introducing the spaces of  symbols and give a detailed description of  Toeplitz operators in the Fock space. Then, in Section 3, we give the proof  of the finite rank theorem with functions serving as symbols.

In another paper the second author plans to present the proof 
of the finite rank result for symbols-distributions without the 
compact support condition, belonging to a certain class defined 
by restrictions on the behavior at infinity.
It is known, see, e.g., \cite{RozToepl}, that the finite rank property, once established for a Bergman type space of analytical functions, can be extended to some other Bergman type spaces. There are some specifics of that procedure when the compactness condition is dropped. This kind of questions will be dealt with elsewhere.

G.R. expresses his gratitude to the  Mittag-Leffler Institute where he was given an excellent possibility to work on the paper.

\section{Toeplitz operators in the Fock space. Classes of symbols}\label{Sect2}
\subsection{Operators with bounded symbols}
We start this section by recalling some basic facts concerning the Fock space and operators there.

We identify the plane $\R^2$ with the complex plane $\C$ and denote by $\n$ the normalized Gaussian measure, $d\n= \o(z) d\l$, where $d\l$ is the Lebesgue measure, $\o(z)=\pi^{-1} e^{-|z|^2}.$ (We choose this version of the weight, rather than the alternative one $(2\pi)^{-1}e^{-|z|^2/2}$ in order to be in conformity with \cite{BauLe}.) In the space $\Hc=L_2(\C,d\n)$ we consider the subspace $\Bc$, the Fock space, which consists of entire  functions. We denote by $(\cdot,\cdot)$ the scalar product in this space. The orthogonal projection $\Pb:\Hc\to \Bc$ is known to be an integral operator with smooth kernel,
\begin{equation}\label{2.kernel}
    (\Pb u)(z)=\int_{\C}\k(z,w){u(w)}d\n(w)=(u,\k(\cdot,z)),
\end{equation}
where
$\k(z,w)=e^{z\overline{w}}$. In particular, if $u\in\Bc,$ we have $\Pb u =u$, or
\begin{equation}\label{2.repro}
    u(z)=\int_\C \k(z,w){u(w)}d\n(w)=(u,\k(\cdot,z));\end{equation}
the latter equation is called the \emph{reproducing property} and $\k(z,w)$ is called the reproducing kernel.

For a  function $F$ defined on $\C$, the Toeplitz operator with symbol $F$ acts as an integral one,
\begin{equation}\label{2.Oper}
    (\Tb(F)u)(z)=(\Pb F u)(z)=\int_{\C}\k(z,w)F(w){u(w)}d\n(w),
\end{equation}
defined on such functions $u\in\Bc$ for which $\k(z,.) F u\in \Hc$.
If $F\in  L^\infty(\C)$, this operator is, obviously, everywhere defined and bounded in $\Bc$, as a product of bounded operators. The operator's sesquilinear  form is
\begin{equation}\label{2.form}
    \tb_F(u,v)=(\Tb(F)u,v)=\int_C F(w)u(w)\overline{v(w)}d\n(w).
\end{equation}
\subsection{Operators with unbounded symbols}
Our aim now is to define the operator for a larger class of symbols. There are several discussions of this topic in the literature, see, e.g., \cite{BauLe}, \cite{Ja1}, \cite{RozToepl}, \cite{RozSW} and references therein.

If we drop the boundedness condition for $F$, the Toeplitz operator is not necessarily bounded; it is defined on the set of functions $u\in\Bc$ satisfying  $\Tb(F)u\in \Bc.$ As in \cite{BauLe}, we introduce classes $\DF_\cF$, $\cF< 1$ by
\begin{equation}\label{1.Dc}
    \DF_\cF=\{F: \sup_{z\in\C}|F(z)|e^{-\cF|z|^2}<\infty\}.
\end{equation}
Furthermore, $\DF_{1,-}$ is the space of functions $F$ satisfying $|F(z)|=O(e^{|z|^2-a|z|})$, $|z|\to\infty$ for every $a<\infty$.

Generally, it is hard to describe explicitly the domain of the Toeplitz operator for an unbounded symbol.
If $F\in \DF_\cF$, $\cF<1/2$, then the domain of $\Tb(F)$ contains
all functions $u\in \Bc\cap \DF_{1/2-\cF}$ and, in particular, is
dense in $\Bc$. Under a less restrictive condition, $F\in \DF_{1,-}$,
the Toeplitz operator is still densely defined and, in particular, its
domain contains all analytical polynomials, as well as all
reproducing kernels $\k_z={\k(\cdot,z)}$, $z\in\C$ and their finite
linear combinations.  In the finite rank problem which we mainly
discuss in this paper, it is sufficient to consider the action of the
operator on these dense subsets.
Different extensions of the operator $\Tb(F)$ beyond \eqref{1.Dc}
are discussed in \cite{Ja1}, \cite{Ja2}, and in \cite{BauLe}.

In the analysis of the finite rank problem, it is  convenient to consider the sesquilinear form $\tb=\tb_F$,
 \begin{equation}\label{2.sesqForm}
    \tb_F(u,v)=(F,\bar{u}v)=\int_{\C} F(z){u(z)}\overline{v(z)} \o(z)d\l(z).
 \end{equation}

  If $F\in \DF_\cF, \cF<1$, then this form is defined at least on  functions $u,v\in \Bc\cap\DF_{\cF'}$, with $\cF'<(1-\cF)/2$.  The latter set is, again, dense in $\Bc$. If $F$ is a real function of  constant sign, then the sesquilinear form thus defined is closable and corresponds to a self-adjoint operator. In the general case, there is no natural way to associate a closed operator with the sesquilinear form \eqref{2.sesqForm}.  However, for $F\in \DF_{1,-}$, the sesquilinear form \eqref{2.sesqForm} is consistent with the action of the operator $\Tb(F)$ at least on the functions $u,v$ being  the reproducing kernels, $u=\k_z=\k(\cdot,z)$, $v=\k_{z'}=\k(\cdot,z')$, or analytical polynomials $u=p$, $v=q$, $p,q\in\mathcal P$ :
  \begin{equation}\label{2.form2}
    (\Tb(F)\k_z,\k_{z'})=\tb_F(\k_z,\k_{z'})
 \end{equation}
 and $(\Tb(F)p,q)=\tb_F(p,q)$.

\subsection{Finite rank operators and forms}
We start by recalling that an everywhere defined operator $\Tb$ in the Hilbert space $\Kc$ is said to be of finite rank if for some elements $f_j,g_j\in\Kc, j=1,\dots,N$
\begin{equation}\label{2.FRbounded}
    \Tb u=\sum_{j=1}^N (u,g_j)f_j
\end{equation}
for all $u\in\Kc.$ As usual, it is much more convenient to use the sesquilinear form in the study of the properties of operators. In the language of sesquilinear forms, equivalently,
\begin{equation}\label{2.FRforms}
    (\Tb u,v)=\sum_{j=1}^N (f_j,v)(u,g_j)
\end{equation}
for all $u,v\in\Kc.$ The smallest possible number $N$ in such representations is called the rank of the operator. For uniformity, we say that the zero operator and only it has rank 0, i.e., the sum on the right in \eqref{2.FRbounded}, \eqref{2.FRforms} is empty.
By \eqref{2.FRbounded}, a finite rank operator is automatically bounded.

We will consider a more general case, when equation
\eqref{2.FRforms} holds not for all $u,v\in \Kc$ but only for $u,v$ in some subset $\Kc^0\subset\Kc$. If, still, $f_j,g_j\in\Kc$ and $\Kc^0$ is dense in $\Kc$, these two definitions are equivalent, by  continuity. We, however, consider the cases when the representation \eqref{2.FRforms} holds with $f_j,g_j\not\in\Kc$
(with corresponding modification of the pairing).

We denote by $\Bc^\circ$ the space of entire analytical functions belonging to  $\DF_{1,-}$, $\Bc^\circ=\Ac\cap \DF_{1,-}$. For $f\in\Bc^\circ$ and $v$ of exponential growth, the expression $(f,v)$ is correctly defined, although $f$ is not necessarily in $\Bc:$
\begin{equation}\label{2.(fv)}
    (f,v)=\int_{\C}f(w)\o(w)\overline{v(w)}d\l(w),\qquad (v,f)=\overline{(f,v)},
\end{equation}
and this definition is consistent with the definition of the scalar product in the space $\Bc$. In particular, $(f,v)$ is defined, as in \eqref{2.(fv)},   for $v$ being an analytical polynomial or the reproducing kernel. By continuity,
the reproducing relation \eqref{2.repro} extends to all $f\in \Bc^\circ:$
\begin{equation}\label{2.repr.ext}
    (f, \k_z)=f(z).
\end{equation}
Note also that the latter equation admits differentiation in $z$, since the
derivative of $\k(z,\cdot)$ is, again, of exponential growth: $\partial_z^\a\k(z,w)=\bar w^\a\k(z,w)$.

Now we can give a definition of more general finite rank operators and forms.
\begin{definition}\label{2.FRdefGen}Let $F$ be a function in $\DF_{1,-}$. We say that the sesquilinear  form $\tb=\tb_F,$ defined in  \eqref{2.sesqForm}, \emph{has finite rank on reproducing kernels} if, for some functions $f_j,g_j\in \Bc^\circ$,
\begin{equation}\label{2.repr.gener}
  \tb(u,v)=\sum_{j=1}^N(u,{g_j})(f_j,v),
\end{equation}
for all $u,v$ being reproducing kernels, i,e., $u=\k_z, v=\k_{z'}$. In other words,
\begin{equation}\label{2.repr.frank}
    \tb(\k_z,\k_{z'})=\sum_{j=1}^N(\k_z,{g_j})(f_j,\k_{z'}).
\end{equation}
\end{definition}
For $f_j,g_j\in \Bc$, this definition is consistent with \eqref{2.FRforms}. However, for $f_j,g_j$ outside $\Bc$, the functionals on the right hand side in \eqref{2.repr.gener} are not continuous with respect to $u,v$ in the space $\Bc$, and therefore, a sesquilinear form $\tb$ is not necessarily \emph{a priori} bounded on $\Bc$. Such boundedness will only follow \emph{post factum} from the finite rank theorems of this paper.

In a similar way, we say that the sesquilinear form has finite rank on polynomials, if for some functions $f_j,g_j\in \Bc^\circ$ we have
\begin{equation}\label{2.repr.polyn}
    \tb(z^k,z^{k'})=\sum_{j=1}^N(z^k,g_j)(f_j,z^{k'}), \qquad k,k'\in\mathbb Z_+.
\end{equation}

It is easy to see that these two properties are equivalent. In one direction it follows from the relation $w^k=\partial^k_{\bar{z}} \k(w,z)|_{z=0}$, in the other direction, it follows from the Taylor expansion for  $\k(z,w)$.

The finite rank property of the sesquilinear form is closely related 
to some properties of infinite matrices. For an infinite system of distinct points $z_j\in \C$, we consider the matrix $\KF$ with elements $\kF_{k,l}=\tb(\k_{z_k},\k_{z_l})$.
Another infinite matrix, $\PF$, associated with the sesquilinear form $\tb$, is defined by setting $\pF_{k,l}=\tb(z^k,z^l)$. If the sesquilinear form has finite rank on polynomials (on reproducing kernels), then the matrices $\KF$ and $\PF$ have the same finite rank.  In \cite{RShir} the converse statement has been proved for $F$ with compact support. In the general case, it is unclear whether the converse statement is true.


\section{The finite rank theorem}\label{3.Section}
In this section we present a proof of the finite rank theorem,
using an approach different from that in \cite{Lue2}, for the symbols  $F$ being \emph{functions} in $\DF_{1,-}$. We do not assume that the symbols have compact support. This approach, actually, has been under discussion for a certain time, it has roots in the failed proof in \cite{Lue1} and in the rank one proof mentioned in \cite{Cob}. Explicitly, the reduction of the finite rank problem to a statement about analytical functions (however, for the case of bounded symbols) is described in \cite{BauLe},  see the discussion around Question D there.
\begin{theorem}\label{3.Th.Lu form} Suppose that the symbol $F$ is a function in $\DF_{1,-}$ and that the sesquilinear form $\tb(u,v)=(Fu,v)=(F,\bar{u}v)$ has finite rank on reproducing kernels. Then $F=0$.
\end{theorem}

\begin{proof} By the definition of finite rank sesquilinear forms,  there exist functions $f_j,g_j\in \Bc^0,\ j=1,\dots, N$ such that

\begin{equation}\label{ThL.2}
    \tb (u,v)=\sum_{j=1}^N (u,g_j)(f_j,v)
\end{equation}
for all $u,v$ being the reproducing kernels at some points.
Following \cite[Lemma 5.2]{BauLe} we make a special choice of $u,v$ in \eqref{ThL.2}: $u=\k_{iz}$, $v=\k_{-iz}$, $z\in \C$. Using the reproducing property \eqref{2.repro}, we obtain
\begin{multline}\label{ThL.3}
  \F(z) := \tb(\k_{iz},\k_{-iz})=\sum_{j=1}^N (\k(iz,.),\overline{g_j})(f_j,\k(-iz,.))\\=\sum_{j=1}^N \overline{g_j(iz)}f_j(-iz).
\end{multline}
On the other hand,
\begin{gather}
    \F(z) = \tb(\k_{iz},\k_{-iz})=(F\k_{iz},\k_{-iz})=\int_{\C}
    F(w)\exp(-iz\overline{w}-iw \overline{z})d\n(w)\nonumber
    \\
    =\int_{\C}F(w)\o(w)\exp(-2i\re(z\overline{w}))d\l(w)=\FF(F\o)(2x),
    \,\,\,\, x\in \R^2,\label{ThL.4}
\end{gather}
where   $\FF$ is the Fourier transform in $\R^2$ (here we identify the complex plane  $\C=\C_z$ with the real plane $\R^2=\R^2_x$, $x=(x_1,x_2), z=x_1+ix_2$).
For derivatives $\partial^{\a,\b}\F(z):=\partial^\a\dbar^\b \F(z)$ we have
\begin{gather}
  \partial^{\a,\b}\F(z)=  (-i)^{\a+\b}\int \int_{\C}F(w)\o(w)\overline{w}^\a w^{\b} \exp(-2i\re(z\overline{w})d\l(w)\nonumber \\
  =(-i)^{\a+\b}\FF(F(w)\o(w)\overline{w}^\a w^{\b})(2x).\label{ThL.5}
\end{gather}
By the properties of the Fourier transform, since $F(w)\o(w)$ decays exponentially, and all functions $F(w)\o(w)\overline{w}^\a w^{\b}$ belong to $L^1(\R^2,d\l)$, we obtain that the functions $\partial^{\a,\b}\F(z)$ decay at infinity for all $\a,\b$.

Next, using \eqref{ThL.3} we calculate the derivatives of $\F$:
\begin{multline}
   \partial^{\a,\b}\F(z)=\partial^\a \dbar^\b\sum_{j=1}^N \overline{g_j(iz)}f_j(-iz)\\=\sum_{j=1}^N{(-i)}^{\a+\b}f_j^{(\a)}(-iz)\overline{g_j^{(\b)}(iz)}.\label{ThL.6}
\end{multline}
So, by the decay of these derivatives, we have
\begin{equation}\label{ThL.7}
    \left| \sum_{j=1}^N f_j^{(\a)}(-iz)\overline{g_j^{(\b)}(iz)}\right|\le \e(z)
\end{equation}
for all $\a,\b<N$, with some function $\e(z)$ tending to zero at infinity.

Now we introduce  vector-functions $\Fb_j:\C\to \C^N, \ j=0,\dots, N-1$, by
$\Fb_j(z)=(f_1^{(j)},\dots,f_N^{(j)})$. For any fixed $z\in \C$ we consider a convex compact set $K^{\Fb}_z\subset \C^N$
invariant with respect to componentwise multiplication by the phase factor, $(z_1,\dots,z_N)\mapsto (e^{i\theta}z_1,\dots,e^{i\theta}z_N)$, $\theta\in[0,2\pi)$, and
generated by the vectors $
\Fb_j(z), j=0,\dots,N-1$. The volume of this compact is equivalent  (up to a factor
depending only on $N$) to the square of the absolute value of the determinant of the matrix $\Fb$,
\begin{equation*}
\Fb=    \left(\begin{array}{c}
       \Fb_0(z) \\
       \dots \\
       \Fb_{N-1}(z)
     \end{array}\right),
\end{equation*}
which is exactly the square of the modulus of the Wronskian 
\newline\noindent $W(f_1,\dots,f_N)$ of the system of functions $f_j, j=1,\dots,N$.

Indeed, this volume is equivalent to that of the parallelepiped, centered at the origin and
generated by $F_j,iF_j$, $0\le j<N$, in $\mathbb C^N=\mathbb R^{2N}$,
which is equal to the square of modulus of the determinant of the matrix
\begin{equation*}
    \left(\begin{array}{cc}
       \mathop{\rm Re\,} \Fb& \mathop{\rm Im\,} \Fb \\
       \mathop{\rm -Im\,} \Fb& \mathop{\rm Re\,} \Fb
     \end{array}\right).
\end{equation*}
Furthermore,
\begin{multline*}
\mathop{\rm det\,}\left(\begin{array}{cc}
       \mathop{\rm Re\,} \Fb& \mathop{\rm Im\,} \Fb \\
       \mathop{\rm -Im\,} \Fb& \mathop{\rm Re\,} \Fb
     \end{array}\right)=
     \mathop{\rm det\,}\left(\begin{array}{cc}
       \mathop{\rm Re\,} \Fb+i\mathop{\rm Im\,} \Fb& \mathop{\rm Im\,} \Fb-i\mathop{\rm Re\,} \Fb \\
       \mathop{\rm -Im\,} \Fb& \mathop{\rm Re\,} \Fb
     \end{array}\right)\\=\mathop{\rm det\,}\left(\begin{array}{cc}
       \mathop{\rm Re\,} \Fb+i\mathop{\rm Im\,} \Fb& 0 \\
       \mathop{\rm -Im\,} \Fb& \mathop{\rm Re\,} \Fb-i\mathop{\rm Im\,} \Fb
     \end{array}\right)=|\mathop{\rm det\,}(
       \mathop{\rm Re\,} \Fb+i\mathop{\rm Im\,} \Fb)|^2.
\end{multline*}

In a similar way, we associate with the system of functions $g_j$ the mappings $\Gb_j=(\overline{g_1^{(j)}},\dots,\overline{g_N^{(j)}})$ and the compact $K_z^{\Gb}$ with volume equivalent to the square of the modulus of the Wronskian $|W(g_1,\dots,g_N)|$.

By taking various convex combinations of inequality \eqref{ThL.7}, written for different $\a,\b$, we obtain that
\begin{equation}\label{ThL.8}
|X Y|\le \e(z)
\end{equation}
for any $X\in K^{\Fb}_z, Y\in K^{\Gb}_z$. In the geometrical language, \eqref{ThL.8} means that the convex set $K^{\Gb}_z$ is contained in the \emph{polar set}  of the  compact $\e(z)^{-1}K^{\Fb}_z$, i.e., $K^{\Gb}_z\subset (\e(z)^{-1}K^{\Fb}_z)^{\circ}$.

Now we apply the so called Santal\'o's inequality (for the precise
formulation we use here see \cite[Corollary 3.3]{Cor}): for every convex compact $K\in\C^N$ invariant with respect to componentwise multiplication by the phase factor, we have
\begin{equation}\label{ThL.9}
    \vol(K)\vol(K^\circ)\le C(N).
\end{equation}
By \eqref{ThL.8}, \eqref{ThL.9}, we obtain
\begin{equation*}
    \vol(K^{\Fb}_z)\vol(K^{\Gb}_z)\le C(N)\e(z),
\end{equation*}
and, therefore,
\begin{equation*}
    |W(f_1,\dots,f_N)W(g_1,\dots,g_N)|^2\le C(N)\e(z).
\end{equation*}
Since  $W(f_1,\dots,f_N)W(g_1,\dots,g_N)$ is an entire function in $\C^N$, its decay at infinity  implies, by Liouville's theorem, that it is identically zero, and therefore one of the factors above should be zero. If, say, $W(f_1,\dots,f_N)\equiv 0$, the function $f_N$ lies in the linear span of the functions $f_1,\dots,f_{N-1}$, so the rank of the operator is, in fact, $N-1$. The induction on the rank concludes the proof.
\end{proof}

If, in particular, $F$ is an arbitrary bounded function, then the Toeplitz operator can be defined in the usual way, it is associated with the sesquilinear form, and we   arrive at the resolution of the finite rank problem in its initial setting.

\begin{corollary} Suppose that $F$ is a bounded function in $\C$. If the Toeplitz operator $\Tb(F)$ in $\Bc$ has finite rank, then $F=0$.
\end{corollary}

\begin{remark} The condition on the growth  of  $F$, i.e., the
requirement that $F\in\DF_{1,-}$, imposed in Theorem
\ref{3.Th.Lu form}, is almost sharp. It is easy to construct examples
(see, e.g., \cite{BauLe}, or \cite{GrVas}) of nontrivial functions $F$
such that $\o(w)F(w)$ decay at infinity  just a bit slower than
exponentially, but even the zero rank statement is wrong.
\end{remark}

\begin{remark} In \cite{RShir}, \cite{AlexRoz} and \cite{BRChoe1}, the finite rank theorem of Luecking has been extended to the multidimensional case.  The method used in \cite{RShir} and \cite{AlexRoz} is not tied to a particular method of proving the one-dimensional version of the theorem, and, with some minor modifications, can be adapted to prove the extension of Theorem \ref{3.Th.Lu form} to the case of Toeplitz operators in the
multidimensional Fock space. We do not go into details here since there exist a more direct way of establishing the finite rank property in $\C^d$. Namely,  Proposition 5.6 in \cite{BauLe} reduces the finite rank problem in any dimension to the one in $\C^1$, for the symbol being a function in $\DF_\cF$, $\cF< 1$,
and a simple modification of the argument used there gives the same reduction for the symbols in the class $\DF_{1,-}$. Thus, the multi-dimensional generalization of Theorem \ref{3.Th.Lu form} for symbols-functions  follows.
\end{remark}

\end{document}